\documentclass[a4paper]{article}
\usepackage{amsmath, amsthm, amssymb, graphicx, enumitem, subcaption, placeins}
\graphicspath{{./drawings/}}

\newcommand{\cmp}{^{\scriptscriptstyle\complement}}
\newcommand{\abs}[1]{\lvert#1\rvert}
\newcommand{\uxd}[3][\bigtriangleup]{U^*_{#2}#1 U^*_{#3}}

\newcommand{\uxge}[2]{\uxd[\supseteq]{#1}{#2}}

\newcommand{\uxle}[2]{\uxd[\subseteq]{#1}{#2}}
\newcommand{\uxm}[2]{\uxd[\setminus]{#1}{#2}}
\newcommand{\uxeq}[2]{\uxd[=]{#1}{#2}}
\newcommand{\niu}[1]{\abs{N_i\cap U^*_{#1}}}
\newcommand{\ie}{i.e.\ }
\newcommand{\eg}{e.g.\ }
\newcommand{\seq}[3][1]{#2_{#1},\ldots,#2_{#3}}
\newcommand{\bseq}[3][1]{(\seq[#1]{#2}{#3})}
\newcommand{\lemma}[1]{Lemma~\ref{#1}}
\newcommand{\theorem}[1]{Theorem~\ref{#1}}
\newcommand{\fig}[1]{Figure~\ref{#1}}

\newtheorem{thm}{Theorem}
\newtheorem{lem}[thm]{Lemma}
\newtheorem{prop}[thm]{Proposition}

\makeatletter
\let\@fnsymbol\@arabic
\makeatother

\newcommand\thankx[1]{\begingroup\let\rlap\relax\thanks{#1}\endgroup}

\begin{document}

\title{Majority dynamics with one nonconformist}
\author{John Haslegrave\thankx{University of Warwick, Coventry, UK. {\tt j.haslegrave@cantab.net}}, Chris Cannings\thankx{University of Sheffield, Sheffield UK. {\tt c.cannings@sheffield.ac.uk}}}
\maketitle

\begin{abstract}
We consider a system in which a group of agents represented by the vertices of a graph synchronously update their opinion based on that of their neighbours. If each agent adopts a positive opinion if and only if that opinion is sufficiently popular among his neighbours, the system will eventually settle into a fixed state or alternate between two states. If one agent acts in a different way, other periods may arise. We show that only a small number of periods may arise if natural restrictions are placed either on the neighbourhood structure or on the way in which the nonconforming agent may act; without either of these restrictions any period is possible.

\textit{Keywords: majority dynamics; threshold automata; voter model; social learning; periodicity.}
\end{abstract}

\section{Introduction}

We consider a general setting in which a number of agents with a system of neighbourhood relationships have binary opinions which they update synchronously based on their neighbours' opinions. Neighbourhood is an arbitrary symmetric relation, and we represent the agents as vertices of a graph with edges, and, if necessary, loops, corresponding to the neighbourhood relation. Perhaps the most natural model for such updating of opinion is for each agent to adopt the more popular opinion among its neighbours (majority dynamics). A more general model along the same lines is to allow each agent to be inclined against a particular opinion, only adopting that opinion if sufficiently many neighbours (not just a simple majority) do. Different agents can be inclined toward different opinions or to different degrees. Such a system forms a threshold network; threshold networks were introduced by McCulloch and Pitts \cite{MP43} to model activation of neurons. They also arise naturally as myopic best response strategies in networks of agents playing a coordination game (see \eg \cite{BS96}, \cite{SC13}).

Majority dynamics and the more general threshold networks have been much studied. A classical result is the period-2 property. Since any finite threshold network has only a finite number of states and the progression from state to state is deterministic and memoryless, periodic behaviour must eventually arise from any possible starting state. What lengths of period are possible? It is not obvious that there is any constant bound on the period, but in fact the only possible periods are $1$ and $2$. This was proved independently by Goles and Olivos \cite{GO81} (see also \cite{Gol87}) and Poljak and S\^{u}ra \cite{PS83}. Poljak and Turz\'ik \cite{PT86} gave good bounds on the time until periodic behaviour begins.

If the network is infinite then the system does not necessarily reach a periodic state. Furthermore, even if it does, any period can occur. Moran \cite{Mor95} showed that with the additional conditions of bounded neighbourhoods and subexponential growth, both of which are necessary, again only periods 1 and 2 are possible. Ginosar and Holzman \cite{GH00} show that under suitable conditions on an infinite graph a local period-2 property holds, in that each agent will eventually have a constant or alternating opinion (though the system as a whole may never become periodic since the times at which agents settle into these patterns could be unbounded).

Other facets of majority dynamics have been studied, such as the question of whether a bias in the initial opinions tends to be preserved by this process (Tamuz and Tessler, \cite{TT15}), and the threshold of initial bias which results in consensus on infinite trees (Kanoria and Montanari, \cite{KM11}). Probabilistic versions of majority dynamics have been studied on highly-structured graphs. A model where agents make synchronous updates to the majority opinion among their neighbours, but occasionally make errors, dates back at least to work by Gray from the 1980s \cite{Gra87}, but a similar model was considered significantly earlier by Spitzer \cite{Spi70}. Most studies on this model are merely computer simulations, but the few rigorous results include Gray's proof that the 1-dimensional version does not have a phase transition \cite{Gra87} and, more recently, the result of Balister, Bollob\'as, Johnson and Walters \cite{BBJW} that if the probability of error is small then the 2-dimensional torus spends almost all its time in a consensus state.

The opposite notion to majority dynamics, where each agent adopts the minority opinion of its neighbourhood, also arises naturally from the myopic best response strategy for a congestion game \cite{SC13}. We may similarly generalise this to an anti-threshold network, where each agent adopts an opinion if it is sufficiently unpopular in the neighbourhood. The period-2 property for finite anti-threshold networks follows immediately from the result on threshold networks.

Cannings \cite{CC09} considered various situations on simple graphs in which there were both majority and minority agents present, showing that cycles of various lengths could occur. He analysed particularly the case of a complete graph, proving that only cycles of length $1$, $2$ and $4$ are possible, with length $4$ only occurring in the special case of having equal numbers of majority and minority agents. A further class of cubic graphs was considered and possible cycle lengths for various numbers of minority agents obtained by direct simulation. A striking feature of these data is that when only one agent makes minority updates while the others make majority updates, only periods $1$, $2$ and $4$ appear. However, this is not true for all graphs (or even all cubic graphs, \eg \fig{cubic5}). 

A consequence of the way the cubic graphs considered in \cite{CC09} are constructed is that they will have no triangles. We show that it is true for all triangle-free simple graphs that only periods $1$, $2$ and $4$ arise for majority dynamics with one additional agent following a different protocol. This result applies in the much more general setting where the nonconforming agent updates his opinion as any function of its neighbours' opinions, not necessarily choosing the minority opinion, and also if triangles are permitted so long as the nonconformist is not part of any triangle. If loops are permitted then there are more possibilities, but we prove that only a few different periods can arise. We also show that if the nonconforming agent does update to the minority opinion of his neighbours, then again only a few different periods can arise, with no restriction on triangles.

We will prove all our results for the general threshold situation, but they could equivalently be re-stated in terms of majority updates. It is easy to see that an agent updating according to an arbitrary threshold may be simulated by a suitable bundle of majority agents, and so any dynamics arising from arbitrary threshold networks with one nonconformist can also arise from majority dynamics with one nonconformist on a larger graph. This larger graph can also easily be chosen in accordance with the various restrictions on graphs that we consider.

Formally, we fix a finite graph $G$, which may have loops but not multiple edges, on vertex set $\{\seq vn\}$. For each $i$, write $N_i$ for the neighbourhood of $v_i$ (including $v_i$ if there is a loop there). The graph is initialised by giving each vertex one of two opinions, which we represent as $\{+1,-1\}$, at time $0$, and all vertices simultaneously update their opinions at each time step. Write $U_t$ for the set of vertices having opinion $+1$ at time $t$. Each vertex $v_i$ has an update rule which depends only on the state of $N_i$ at the previous time step, \ie for each $i$ there is a set system $\mathcal{S}_i\subseteq\mathcal{P}N_i$ such that $v_i\in U_{t+1}$ if and only if $N_i\cap U_t\in\mathcal{S}_i$. We say that $v_i$ has a \textit{threshold rule} with threshold $r_i$ if $\mathcal{S}_i=\{A\subseteq N_i:|A|\geqslant r_i\}$ for some $r_i$, and an \textit{anti-threshold rule} if $\mathcal{S}_i=\{A\subseteq N_i:|A|<r_i\}$. We will always assume that every vertex except $v_1$ has a threshold rule, with $v_i$ having threshold $r_i$ for $i>1$.

\section{A Lyapunov operator}

Proofs of the period-2 property for threshold networks (see \cite{GO81}, \cite{PS83}, \cite{Gol87} and \cite{BS96}) proceed by defining a suitable Lyapunov operator, proving that it is bounded, integer-valued and non-decreasing, so must be ultimately constant, and showing that if at any step the value does not change then the state is identical to the previous state but one. In this section we give a modified Lyapunov operator for the situation where $v_1$ has an arbitrary rule, and show that provided every other vertex has a threshold rule this is still bounded and non-decreasing, and must be an integer multiple of $1/2$, so is ultimately constant. The analysis of what can happen once this operator has reached its final value is much more complicated than for pure threshold networks, and we carry out this analysis separately for triangle-free graphs in Section 3 and for general graphs with $v_1$ having an anti-threshold rule in Section 4.

\begin{thm}\label{lyap}For $t$ sufficiently large, if $v_i\in\uxd{t-1}{t+1}$ then $v_i\in N_1$ and $\niu t=r_i-1$.\end{thm}
\begin{proof}For $i=2,\ldots,n$ set
\[
s_i=\begin{cases}r_i-1&\text{if }v_i\in N_1\\r_i-\tfrac12&\text{if }v_i\not\in N_1\,.\end{cases}
\]
Note that if $i\in U^*_{t+1}$ then $\niu t\geqslant s_i$, if $i\notin U^*_{t+1}$ then $\niu t\leqslant s_i$, and if $v_i\not\in N_1$ then both inequalities are strict. (If $v_i$ is a neighbour of $v_1$ and $\niu t=s_i$ then the opinion of $v_i$ at time $t+1$ equals that of $v_1$ at time $t$.)

For $t>0$ define $x(t)$ to be the number of pairs $(i,j)$ such that $i\in U^*_t$ and $j\in U^*_{t-1}\cap N_i$. Let $y(t)=\sum_{i\in U^*_t}s_i$, and let $z(t)=x(t)-y(t)-y(t-1)$. 

Note that $z(t+1)-z(t)=x(t+1)-x(t)+y(t-1)-y(t+1)$. We may write $x(t+1)$ as $\sum_{i\in U^*_{t+1}}\niu t$ and $x(t)$ as $\sum_{i\in U^*_{t-1}}\niu t$. Consequently
\begin{equation*}
x(t+1)-x(t)=\sum_{\substack{i\in U^*_{t+1}\\i\notin U^*_{t-1}}}\niu t-\sum_{\substack{i\in U^*_{t-1}\\i\notin U^*_{t+1}}}\niu t\,,
\end{equation*}
and so
\begin{equation*}
z(t+1)-z(t)=\sum_{\substack{i\in U^*_{t+1}\\i\notin U^*_{t-1}}} \big(\niu t-s_i\big)+\sum_{\substack{i\in U^*_{t-1}\\i\notin U^*_{t+1}}}\big(s_i-\niu t\big)\,.
\end{equation*}
By our earlier observation, this is a sum of non-negative terms, and so $z(t+1)\geqslant z(t)$ for each $t$. Further, if $v_i$ contributes to either sum then the corresponding term is strictly positive, and so $z(t+1)>z(t)$, unless $v_i$ is adjacent to $v_1$.

Since $2z(t)$ is an integer, and at most $n^2+4n$, $z(t)$ must eventually be constant. Therefore, for $t$ sufficiently large that $z(t)$ has reached its final value, all terms in the sum are zero. Consequently if $v_i\in\uxd{t+1}{t-1}$ then $v_i$ is adjacent to $v_1$ and we must have $\niu t=s_i=r_i-1$.\end{proof}

\section{General rules in triangle-free graphs}

In this section we consider graphs where $v_1$ is not in a non-degenerate triangle. In the loopless case we show that only periods 1, 2 and 4 are possible, but when loops are permitted several other periods may arise. These additional periods do not necessarily require a loop at $v_1$ (\eg \fig{loopv2}).

\begin{thm}\label{tfree}If no two distinct neighbours of $v_1$ are adjacent, and loops are not permitted, then for any update rule at $v_1$ the system reaches a $1$-, $2$- or $4$-cycle.\end{thm}
\begin{proof}By \theorem{lyap}, for sufficiently large $t$ and any fixed vertex $v_i$ which is neither $v_1$ nor adjacent to it, the opinion of $v_i$ is a function of the parity of $t$. If $v_i$ is a neighbour of $v_1$ then the opinion of $v_i$ at any sufficiently large time $t$ depends on the state of the neighbours of $v_i$ at time $t-1$; except for the state of $v_1$ at time $t-1$, all of these depend only on the parity of $t$. Likewise the state of $v_1$ at time $t+1$ depends only on the state of its neighbours at time $t$, which in turn depends only on the state of $v_1$ at time $t-1$ and the parity of $t$. Consequently either $v_1$ is in the same state for every sufficiently large odd $t$ or it alternates between states in successive odd $t$, and the same possibilities apply to sufficiently large even $t$. It follows that $v_1$ is in the same state at times $t$ and $t+4$ for sufficiently large $t$, and therefore that each $v_i$ in the neighbourhood of $v_1$ is in the same state at times $t+1$ and $t+5$. Therefore the state of every vertex repeats after four time steps and the system is in a fixed point or $2$- or $4$-cycle.\end{proof}

\begin{thm}\label{loops}If no two distinct neighbours of $v_1$ are adjacent but loops are permitted, then for any update rule at $v_1$ the system reaches a $1$-, $2$-, $3$-, $4$-, $6$-, $8$-, $10$- or $12$-cycle.\end{thm}
\begin{proof}First we deal with the case where the state of $v_1$ is either ultimately constant or ultimately alternating. In that case, for fixed $i$ and sufficiently large $t$, the state of $v_i$ at time $t$ depends only on the parity of $t$ and the state of $v_i$ at time $t-1$. Moreover, since it plays a threshold rule, changing its state at $t-1$ cannot change its state at $t$ in the opposite direction. So either at all sufficiently large odd $t$ it is a fixed state, or at all sufficiently large odd $t$ it is the same state as at $t-1$, and likewise for even $t$. Consequently it is either a fixed state or alternating in state for sufficiently large $t$. Since this is true for every vertex, the system has period 1 or 2. 

Now suppose that the state of $v_1$ is neither ultimately constant nor ultimately alternating. Fix $1<i\leqslant d$, then one of the following is the case for all sufficiently large even $t$, and one is the case for all sufficiently large odd $t$:
\begin{enumerate}[label=(\arabic*)]
\item $v_i\in U_t$;
\item $v_i\in N_i$ and $v_i\in U_t$ iff $v_1\in U_{t-1}$ or $v_i\in U_{t-1}$;
\item $v_i\notin N_i$ and $v_i\in U_t$ iff $v_1\in U_{t-1}$;
\item $v_i\in N_i$ and $v_i\in U_t$ iff $v_1\in U_{t-1}$ and $v_i\in U_{t-1}$;
\item $v_i\notin U_t$.
\end{enumerate}

There are then 21 possibilities for which pair of these rules applies (since it is not possible for (2) or (4) to apply at one parity and (3) at the other). In some cases the behaviour may be simplified. If (1) applies at one parity and (2) at the other, or if (2) applies at both, then in fact (since we are assuming $v_1$ is not ultimately constant) $v_i\in U_t$ for all sufficiently large $t$. Similarly if (4) and (4) or (4) and (5) apply then $v_i\notin U_t$ for all sufficiently large $t$. If (1) applies at one parity and (4) at the other, or (5) at one and (2) at the other, then in fact for the second parity $v_i\in U_t$ iff $v_1\in U_{t-1}$.

For $j\in\{0,1\}$, write $X_j$ for the set of vertices which satisfy (2) for $t\equiv j$ mod 2 and (4) for $t\not\equiv j$. Write $Y_j$ for the set of vertices which satisfy (3) for $t\equiv j$, or which satisfy (2) for $t\equiv j$ and (5) for $t\not\equiv j$, or which satisfy (4) for $t\equiv j$ and (1) for $t\not\equiv j$; we observed above that all such vertices satisfy $v_i\in U_t$ iff $v_1\in U_{t-1}$ for $t\equiv j$. Write $Z_j$ for the set of vertices which satisfy (4) for $t\equiv j$ and (2) for $t\not\equiv j$; note that $Z_j=X_{1-j}$.

If $i\neq 1$ and $v_i\not\in X_0\cup Y_0\cup Z_0$ then $v_i$ is a fixed state for all sufficiently large even $t$. Since we are assuming $v_1$ is not alternating, eventually $v_1\in U_{t-1}$ for some even $t$ or $v_1\not\in U_{t-1}$ for some odd $t$; in either case all vertices in $X_0$ will be the same state at time $t$, and this will remain true at all future times. We will refer to a set being \textit{monochromatic} if all its vertices have the same state; similarly $Y_0$ and $Z_0$ are each monochromatic for all sufficiently large $t$. Also note that, if $t$ is even, then $Z_0\subseteq U_t\Rightarrow Y_0\subseteq U_t \Rightarrow X_0\subseteq U_t$. So there are only eight possible states which occur at arbitrarily large even $t$: either $X_0\cup Y_0\cup Z_0\subseteq U_t$, $X_0\cup Y_0\subseteq U_t$ but $Z_0\subseteq U_t\cmp$, $X_0\subseteq U_t$ but $Y_0\cup Z_0\subseteq U_t\cmp$, or $X_0\cup Y_0\cup Z_0\subseteq U_t\cmp$, for four possibilities, and we may have $v_1\in U_t$ or $v_1\in U_t\cmp$. The states must therefore repeat after at most 16 steps.

In fact not all of these can occur infinitely often. If any of the sets defined above are empty then for some parity of $t$ there are at most 6 possible states, so the period is at most 12. Now suppose that they are all non-empty. If $t$ is sufficiently large we may represent the state of the system by a vector in $\{0,1\}\times\{-1,+1\}^4$, with $\boldsymbol{s}=(j,x,y,z,w)$ representing the state where $t\equiv j$ and $X_j$, $Y_j$, $Z_j$ and $v_1$ having states $x$, $y$, $z$ and $w$ respectively. We know that we are further restricted to states for which $x\geqslant y\geqslant z$. There is some function $f$ mapping, for sufficiently large $t$, the state at time $t$ to the state at time $t+1$. Note that if $\boldsymbol{s}$ is $(0,-1,-1,-1,-1)$, $(0,+1,-1,-1,-1)$, $(0,+1,-1,-1,+1)$, $(0,+1,+1,-1,-1)$, $(0,+1,+1,-1,+1)$ or $(0,+1,+1,+1,+1)$ then $f(\boldsymbol{s})$ is of the form $(1,x,x,x,w)$ for some $x$ and $w$. Consequently the image under $f$ of the set of vectors representing states at even time has size at most $6$ (the above four possibilities together with $f((0,-1,-1,-1,+1))$ and $f((0,+1,+1,+1,-1))$), so at most six states occur at sufficiently large odd time and the period is at most $12$.

If the period, $p$, is odd, then every state which occurs infinitely often does so both at both odd times and even times. So if $t$ is sufficiently large and even, 
\[
Z_0\subseteq U_t\Rightarrow X_0\cup Y_0\subseteq U_t\ \,,
\]
and, since $X_0=Z_1$,
\begin{align*}
Z_0\subseteq U_t &\Rightarrow Z_1\subseteq U_t \\
&\Rightarrow Z_1\subseteq U_{t+p} \\
&\Rightarrow X_1\cup Y_1\subseteq U_{t+p} \\
&\Rightarrow X_1\cup Y_1\subseteq U_t \,.
\end{align*}
Similarly, 
\begin{align*}
Z_0\subseteq U_t\cmp &\Rightarrow X_1\subseteq U_t\cmp \\
&\Rightarrow X_1\subseteq U_{t+p}\cmp \\
&\Rightarrow Y_1\cup Z_1\subseteq U_{t+p}\cmp \\
&\Rightarrow Y_1\cup Z_1\subseteq U_t\cmp \,,
\end{align*}
and, since $X_0=Z_1$,
\begin{align*}
Z_0\subseteq U_t\cmp &\Rightarrow X_0\subseteq U_t\cmp \\
&\Rightarrow Y_0\subseteq U_t\cmp \,.
\end{align*}
Consequently, when $p$ is odd, every state which occurs infinitely often is mono\-chromatic on $X_0\cup Y_0\cup Z_0\cup X_1\cup Y_1\cup Z_1$. There are therefore only 4 possible states, depending on the states of that set and $v_1$, so the only possible odd periods are $1$ and $3$.\end{proof}

\section{Minority rule in general graphs}

In this section we show that natural restrictions on the behaviour of the nonconforming vertex give a finite set of possible periods without any restriction on the graph (other than that the system is finite). Suppose $v_1$ obeys an anti-threshold rule. We consider the possible sequences of states of $v_1$, and show that, once the system has reached a recurrent state, only a few sequences are possible. As a result, we show that only a few different periods can arise from such a system: $1$, $2$, $4$, $5$, $6$ and $10$ if loops are not permitted, and the same periods with the addition of $3$ and $8$ if loops are permitted (in fact a loop at $v_1$ must be present to obtain either of these periods). Throughout this section we assume that the system has already reached a recurrent state, and write $c_t$ for the state of $v_1$ at time $t$. Recall that in this case $\niu{t+1}=r_i-1$ for all $v_i\in\uxd t{t+2}$, and so $\uxle t{t+2}$ if $c_{t+1}=+1$ but $\uxge t{t+2}$ if $c_{t+1}=-1$.

\begin{lem}\label{bwb}If $\bseq[t]c{t+2}=(x,-x,x)$ then $\uxeq{t+1}{t+3}$.\end{lem}
\begin{proof}Without loss of generality we assume $x=+1$. Then $\uxge t{t+2}$, and consequently $\uxge{t+1}{t+3}$, since every vertex has at least as many edges to $U_t$ as to $U_{t+2}$. But also $\uxle{t+1}{t+3}$ since $c_{t+2}=+1$.\end{proof}

\begin{lem}\label{loop12}If $\bseq[t]c{t+4}=(x,-x,x,x,-x)$ then $v_1$ has a loop.\end{lem}
\begin{proof}By \lemma{bwb}, $\uxeq{t+1}{t+3}$. Since $c_{t+4}\neq c_{t+2}$ we must have $N_1\cap U_{t+4}\neq N_1\cap U_{t+2}$. Since $U_{t+4}\bigtriangleup U_{t+2}=\{v_1\}$, we must have $v_1\in N_1$.\end{proof}

\begin{lem}\label{wbb}If $\bseq[t+1]c{t+3}=(-x,x,x)$ then $\uxle t{t+4}$ if $x=+1$ and $\uxge t{t+4}$ if $x=-1$.\end{lem}
\begin{proof}We prove the case $x=+1$; the second case is equivalent by swapping the states. Since $c_{t+2}=+1$, $\uxge{t+3}{t+1}$. Similarly $\uxge{t+4}{t+2}$. If $v_i\in\uxd t{t+2}$ then $v_i$ has $r_i-1$ edges to vertices in $U^*_{t+1}$, so at least $r_i-1$ edges to vertices in $U^*_{t+3}$, and so $v_i\in U^*_{t+4}$. Therefore $\uxge{t+4}t$.\end{proof}

\begin{lem}\label{no4}If $\bseq[t+1]c{t+4}=(-x,x,x,x)$ then $c_{t+5}=-x$.\end{lem}
\begin{proof}Again we may assume $x=+1$. Then $\uxge{t+4}t$ by \lemma{wbb} and since $v_1\in U_{t+4}$, $U_{t+4}\supseteq U_t$. Since $v_1$ obeys an anti-threshold rule, $c_{t+5}\leq c_{t+1}=-1$.\end{proof}

\begin{lem}\label{any21}If $\bseq[t+1]c{t+5}=(-x,x,x,-x,x)$ then $U_{t+6}=U_t$.\end{lem}
\begin{proof}Again we assume $x=+1$. By \lemma{wbb}, $\uxle t{t+4}$. Since $c_{t+5}>c_{t+1}$, and $v_1$ obeys an anti-threshold rule, we must have $\abs{U_{t+4}}<\abs{U_t}$, which is only possible if $\uxeq t{t+4}$ and $c_t=+1$. By \lemma{bwb}, $\uxeq{t+4}{t+6}$. If $c_{t+6}=-1$ then $U_{t+6}=U_{t+4}$, and the system repeats with period $2$, but this contradicts the assumption that we started in a recurrent state. So $c_{t+6}=+1$ and so $U_{t+6}=U_t$.\end{proof}

\begin{lem}\label{no3231}If $\bseq[t+1]c{t+10}=(-x,x,x,x,-x,-x,x,x,x,-x)$ then $c_{t+11}=-1$.\end{lem}
\begin{proof}As usual, assume $x=+1$. We have $\uxge{t+7}{t+3}$, so $\uxge{t+8}{t+4}$ and $\uxge{t+9}{t+5}$ (since $c_{t+7}=c_{t+3}$ and $c_{t+8}=c_{t+4}$). But also $\uxle{t+9}{t+5}$, so they are equal, and so $\uxge{t+10}{t+6}$. Now if $v_i\in\uxm{t+4}{t+6}$ then $\niu{t+9}=\niu{t+5}=r_i-1$ and so $v_i\in U^*_{t+10}$. Consequently $\uxge{t+10}{t+4}$. Now we distinguish two cases. If $c_t=-1$ then $U_{t+10}\supseteq U_t$ so $c_{t+11}=-1$. Otherwise we must have $\uxd[\subset]t{t+10}$ (since $U_t\neq U_{t+4}$) and so $\abs{N_1\cap U_t}\leq\abs{N_1\cap U_{t+10}}$, again giving $c_{t+11}=-1$.\end{proof}

\begin{thm}\label{mingame}If the system starts from a recurrent state and $v_1$ follows an anti-threshold rule then one of the following applies:
\begin{enumerate}[label=(\roman*)]
\item $(c_t)$ is constant and the system has period 1 or 2;
\item $(c_t)$ alternates and the system has period 2;
\item $(c_t)$ repeats the sequence $+1,+1,-1,-1$ and the system has period 4;
\item $(c_t)$ repeats $+1,+1,+1,-1,-1,-1$ and the system has period 6;
\item $(c_t)$ repeats $x,x,x,-x$ for $x=\pm 1$ and the system has period 4;
\item $(c_t)$ repeats $x,x,x,-x,-x$ for $x=\pm 1$ and the system has period 5 or 10;
\item $(c_t)$ repeats $x,-x,-x$ for $x=\pm 1$ and the system has period 3 or 6; or
\item $(c_t)$ repeats $x,-x,-x,x,x,-x,-x,-x$ for $x=\pm 1$ and the system has period 8.
\end{enumerate}
In addition, if loops are not permitted then one of (i)--(vi) must apply.\end{thm}
\begin{proof}If $(c_t)$ is constant then $U_0,U_2,\ldots$ is a monotonic sequence and so eventually constant, so the system has period at most 2. Henceforth we assume $(c_t)$ is not constant, and since we started in a recurrent state both possible values of $c_t$ occur infinitely often. Consider the possible lengths of intervals on which $c_t$ does not change. By \lemma{no4} it is impossible for $\bseq[t+1]c{t+5}$ to equal $(-x,x,x,x,x)$, so no such interval can have length exceeding $3$. 

Suppose two consecutive intervals have length $1$, \ie for some $t$ we have $\bseq[t]c{t+3}=(x,-x,x,-x)$. Then $\uxeq{t+1}{t+3}$ by \lemma{bwb} so the states at $t+1$ and $t+3$ are identical and (ii) applies.

Suppose there are two consecutive intervals of length $3$, say $\bseq[t]c{t+7}=(-1,+1,+1,+1,-1,-1,-1,+1)$. Then $\uxle{t+6}{t+2}$ by \lemma{wbb}, and so also $\uxle{t+7}{t+3}$, since for any $i\neq 1$ $\niu{t+6}\leq\niu{t+2}$. If $v_i\in\uxm{t+3}{t+1}$ then $\niu{t+6}\leq\niu{t+2}=r_i-1$ and so $v_i\not\in U^*_{t+7}$. Thus $\uxle{t+7}{t+1}$, and $c_{t+7}=c_{t+1}$ so $U_{t+7}\subseteq U_{t+1}$; since $v_1$ obeys an anti-threshold rule $c_{t+8}=+1$. It follows that $\uxle{t+8}{t+2}$, so $c_{t+9}=+1$, $\uxle{t+9}{t+3}$, so $c_{t+10}=-1$, and $\uxle{t+10}{t+4}$. Now applying the same arguments with inclusions reversed to $t'=t+3$ gives $\uxge{t'+7}{t'+3}$, \ie $\uxeq{t+10}{t+4}$. Consequently the system has period $6$ and (iv) holds.

Suppose there are two consecutive intervals of length $2$ preceded by an interval of length at least $2$. Without loss of generality we have $\bseq[t]c{t+5}=(-1,-1,+1,+1,-1,-1,+1)$ for some $t$. By \lemma{wbb}, $\uxge{t+4}t$ and so $\uxge{t+5}{t+1}$ and $\uxge{t+6}{t+2}$. However, again by \lemma{wbb}, $\uxle{t+6}{t+2}$, so $U_{t+6}=U_{t+2}$ and (iii) applies. 

By \lemma{any21}, if an interval of length $2$ is followed by an interval of length $1$ then (vii) holds. Suppose (vii) does not hold, but there is some interval of length $1$ followed by one of length $2$, \ie without loss of generality we have $\bseq[t]c{t+4}=(-1,+1,-1,-1,+1)$. Then $c_{t+5}=+1$ (since otherwise (vii) holds). By \lemma{bwb}, $\uxeq{t+1}{t+3}$, and $\uxle{t+5}{t+3}$, so $U_{t+5}\subseteq U_{t+1}$ and so $c_{t+6}\leq c_{t+2}=-1$. Now $c_{t+7}=-1$ since otherwise (vii) holds, $c_{t+8}=-1$ since otherwise (iii) holds, and $c_{t+9}=+1$ by \lemma{no4}. Suppose further that $c_{t+10}=+1$. We cannot have $c_{t+11}=+1$, since that would imply (iv), contradicting recurrence of $U_t$, so the intervals from $t$ onwards begin $1,2,2,3,2$. From now on there cannot be two consecutive intervals of length $2$ (which would imply (iii)) or two consecutive intervals of length $3$ (which would imply (iv)), so intervals of length $3$ and $2$ must alternate until an interval of length $1$ occurs (which it must, by recurrence of $U_t$). But this interval of length $1$ cannot follow one of length $2$ (by \lemma{any21}) or one of length $3$ (by \lemma{no3231}). Consequently, by contradiction, we must have $c_{t+10}=-1$. Now $\uxle{t+8}{t+4}$ and $\uxle{t+9}{t+5}$. If $v_i\in\uxd{t+5}{t+3}$ then $\niu{t+4}=r_i-1$ and so $\niu{t+8}\leq r_i-1$, so $v_i\not\in U^*_{t+9}$. So $\uxle{t+9}{t+3}$, and $\uxeq{t+3}{t+1}$. If $\uxd[\neq]{t+9}{t+3}$ then $\abs{N_1\cap U_{t+9}}<1+\abs{N_1\cap U_{t+3}}$, which is impossible since $c_{t+10}<c_{t+4}$. So $\uxeq{t+9}{t+1}$ and (viii) applies.

If none of (i), (ii), (iii), (iv), (vii) or (viii) hold, then, no interval of length $1$ can be followed or preceded by one of length $2$, and no two consecutive intervals can have the same length. We cannot have three consecutive intervals of lengths $1,3,2$, since subsequent intervals must alternate lengths $3$ and $2$ until another interval of length $1$ occurs, but this cannot follow an interval of length $2$, nor one of length $3$ by \lemma{no3231}. So the only remaining possibilities are that periods of lengths $1$ and $3$ alternate, or that periods of lengths $2$ and $3$ alternate.

In the former case we have $c_{t}$ constant for all $t$ of one parity, say all odd $t$, implying that $U^*_0,U^*_2,\ldots$ is a monotonic sequence. Hence this sequence must be constant, and (v) applies. In the latter case, suppose without loss of generality that $\bseq[0]c4=(+1,+1,-1,-1,-1)$ and this pattern repeats. We have $\uxle40$, $\uxle95$ and $U^*_{10}\subseteq\uxle86$. If $v_i\in\uxm64$ then $\niu5=r_i-1$ and so $\niu9\leq r_i-1$, so $v_i\not\in U^*_{10}$. So $\uxge0{10}\supseteq U^*_{20}\cdots$, and so this sequence is constant, giving (vi).

Thus one of the enumerated situations occurs. If there are no loops, by \lemma{loop12} we cannot have an interval of length $1$ followed by one of length $2$, so (vii) and (viii) are impossible and one of (i)--(vi) occurs in this case.\end{proof}

\section{Graphs achieving these periods}
In this section we show that \theorem{mingame}, \theorem{tfree} and \theorem{loops} are best possible, by giving example graphs to show that all periods mentioned may be attained. We also demonstrate that the restriction in Theorems \ref{tfree} and \ref{loops} that no triangle contains $v_1$ is necessary, by giving a family of graphs on which any period can be obtained with a suitable choice of rule at $v_1$.

It is easy to attain period $1$ or $2$, for example on any bipartite $3$-regular graph in which all vertices except $v_1$ follow the majority rule, by starting all vertices at the same state (period $1$) or all vertices of one part in one state and all vertices of the other part in the other state (period $2$). It is easy to see that each vertex other than $v_1$ will have the desired period unaffected by $v_1$, since its other neighbours will form a majority of one state. Consequently, no matter what rule $v_1$ follows it must also have the same period after the first time step.

Period $4$ is also easy to obtain, as it occurs for the graph consisting of a single edge between minority-rule and majority-rule vertices. This follows (iii) of \theorem{mingame}; we give another period-$4$ example to demonstrate (v) can also occur. 

\FloatBarrier\subsection{Graphs without loops}
Here we give examples to show that the remaining alternatives given by \theorem{mingame} for graphs without loops are all possible. In each case all vertex degrees are odd and every vertex obeys the majority rule except for $v_1$ (indicated by the square) which obeys the minority rule.

\begin{figure}[h]
\begin{minipage}[b]{.5\linewidth}
\centering\includegraphics[scale=0.25]{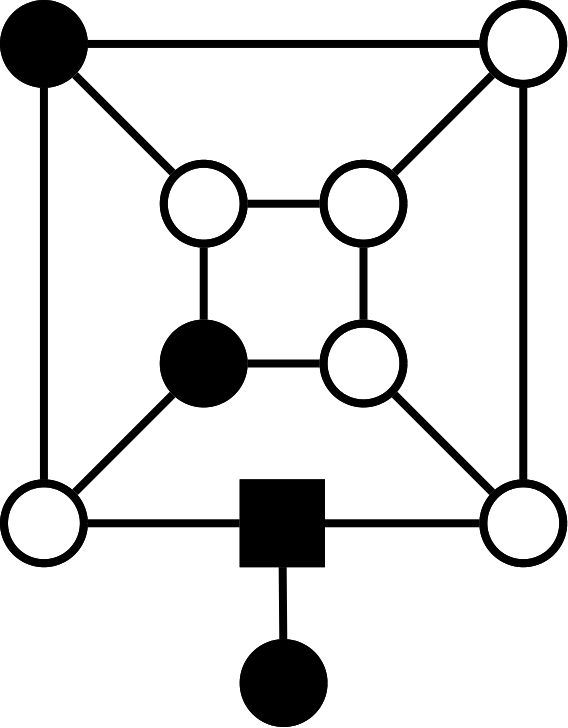}
\subcaption{Period 4 (v)}
\end{minipage}%
\begin{minipage}[b]{.5\linewidth}
\centering\includegraphics[scale=0.25]{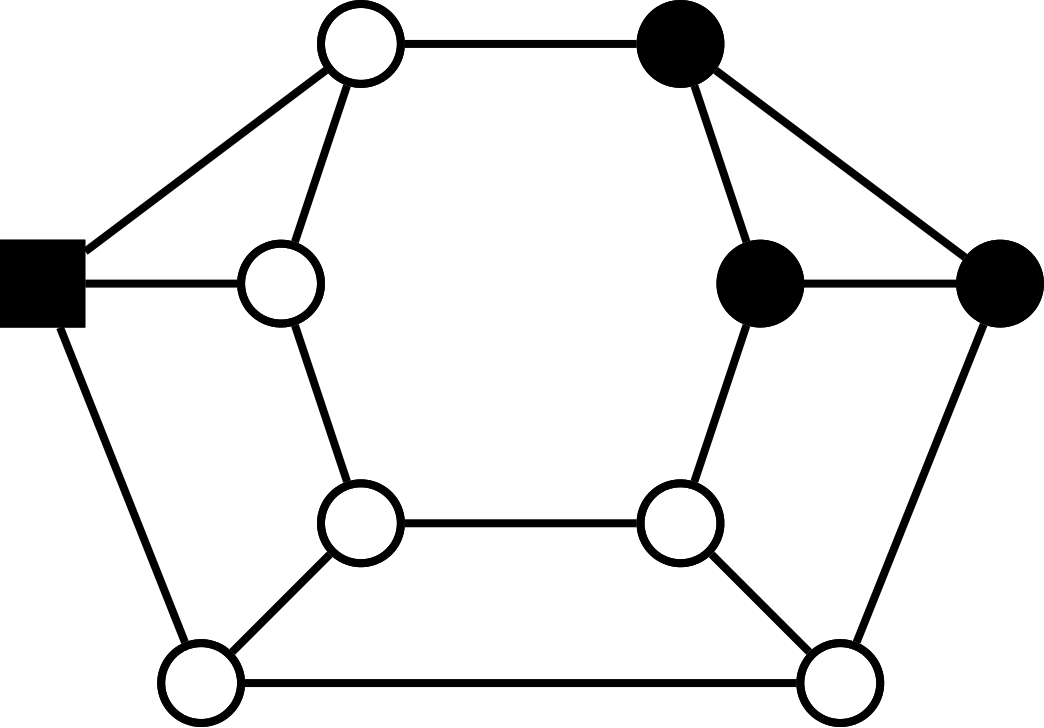}
\subcaption{Period 5}\label{cubic5}
\end{minipage}
\begin{minipage}[b]{.5\linewidth}
\centering\includegraphics[scale=0.25]{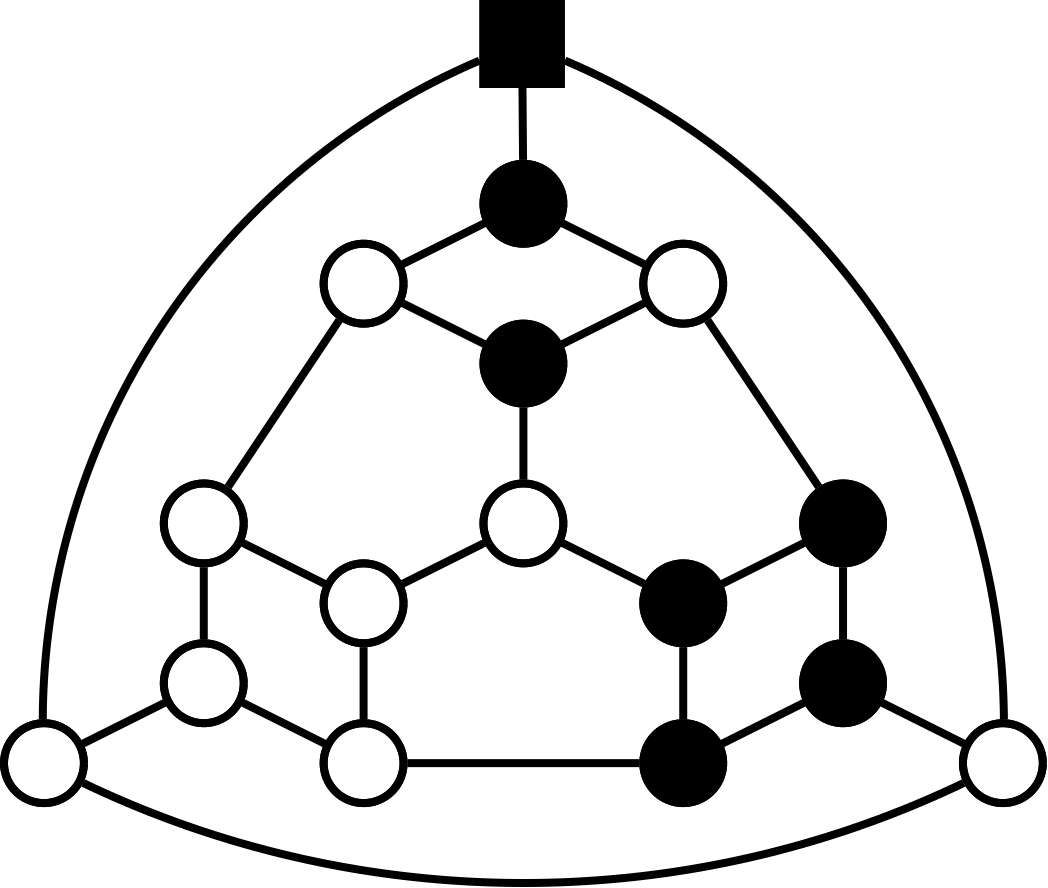}
\subcaption{Period 6}
\end{minipage}%
\begin{minipage}[b]{.5\linewidth}
\centering\includegraphics[scale=0.25]{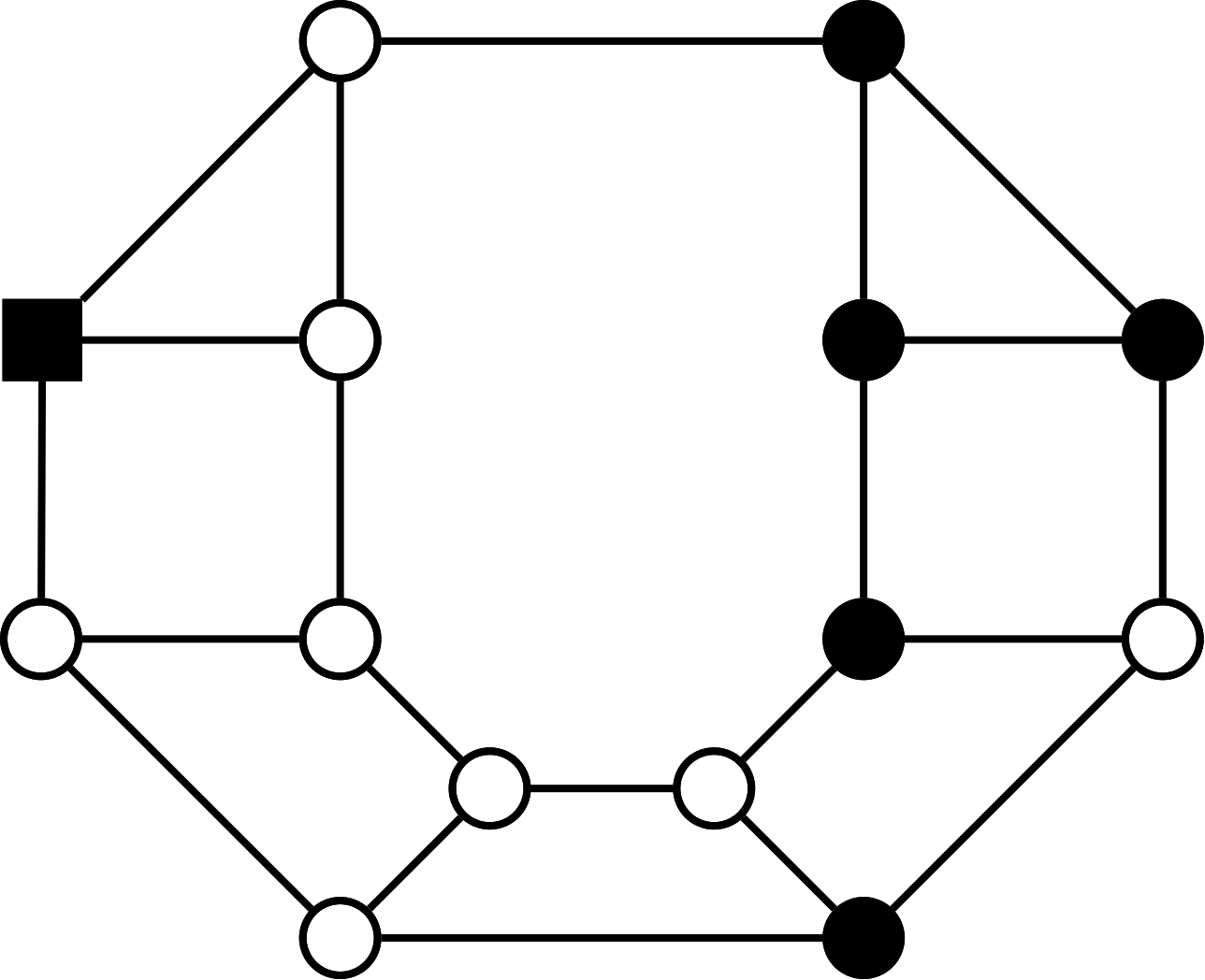}
\subcaption{Period 10}
\end{minipage}
\caption{Loopless graphs with minority rule at $v_1$}\label{fig:1}
\end{figure}

\FloatBarrier\subsection{Triangle-free graphs with loops}

\begin{figure}[h]
\begin{minipage}[b]{.5\linewidth}
\centering\includegraphics[scale=0.25]{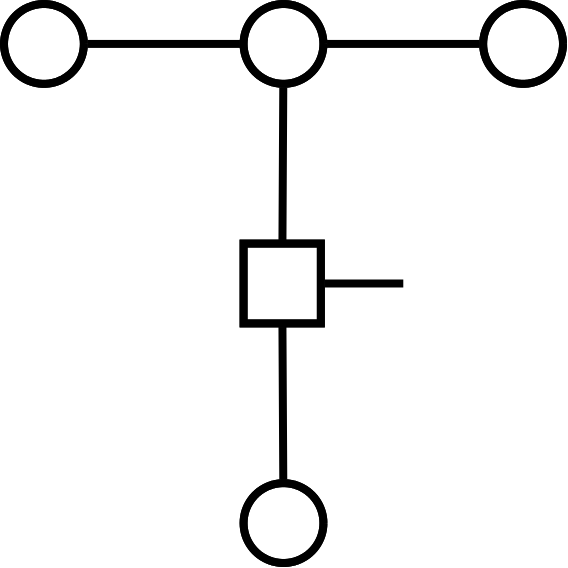}
\subcaption{Period 3}\label{three}
\end{minipage}%
\begin{minipage}[b]{.5\linewidth}
\centering\includegraphics[scale=0.25]{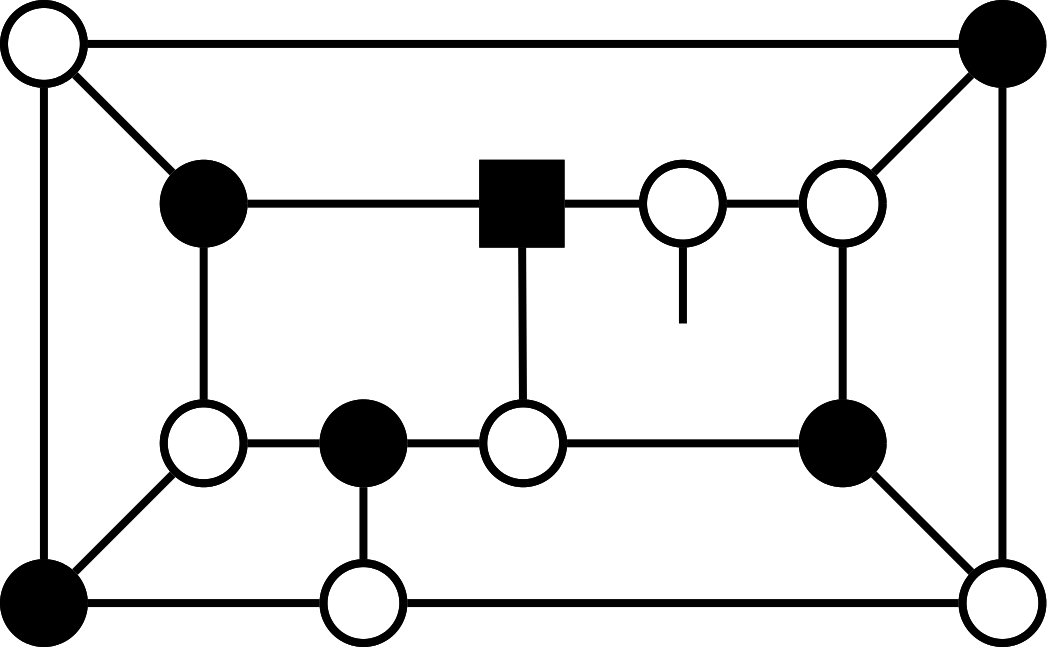}
\subcaption{Period 6}\label{loopv2}
\end{minipage}
\begin{minipage}[b]{.5\linewidth}
\centering\includegraphics[scale=0.25]{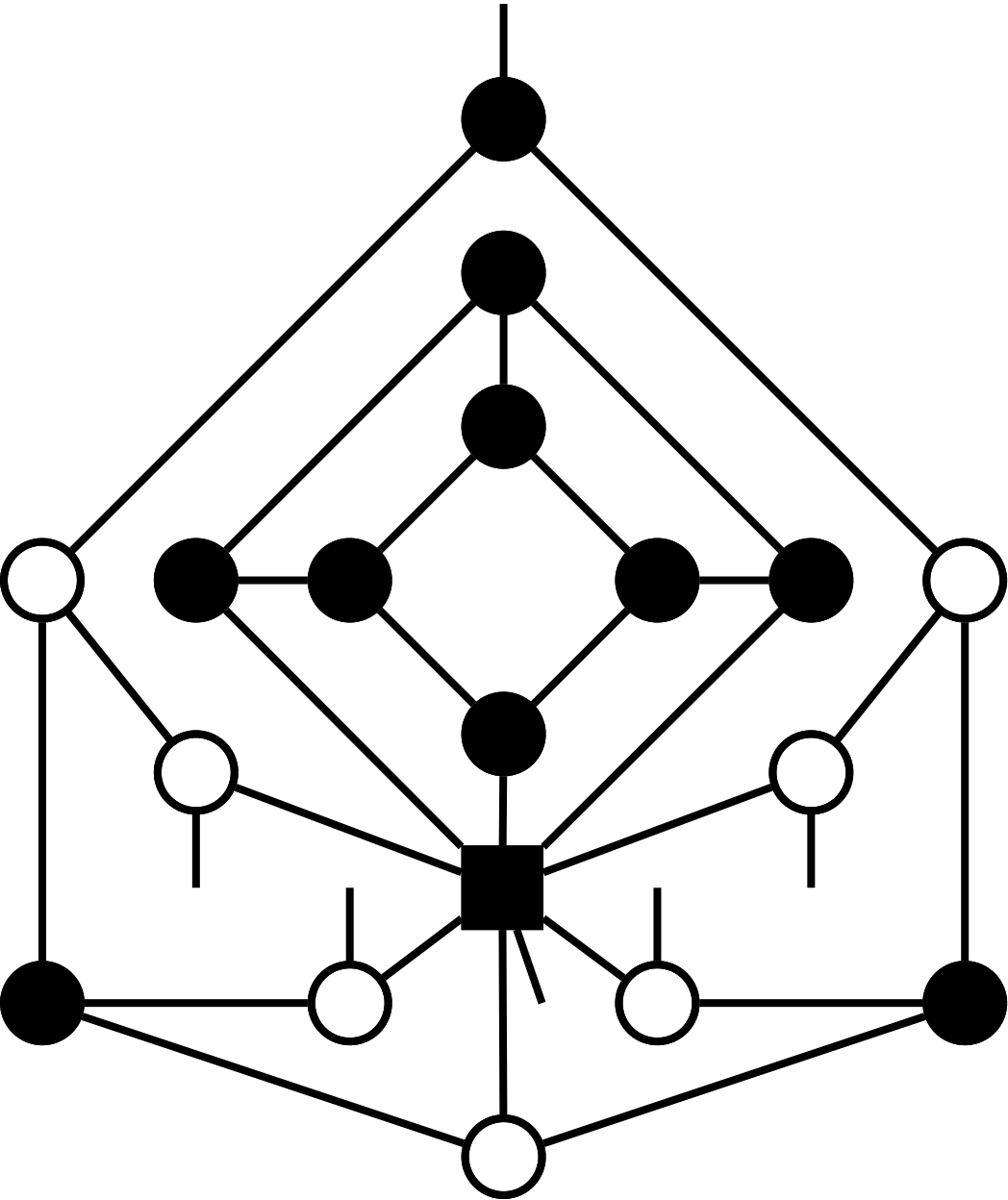}
\subcaption{Period 8}
\end{minipage}%
\begin{minipage}[b]{.5\linewidth}
\centering\includegraphics[scale=0.25]{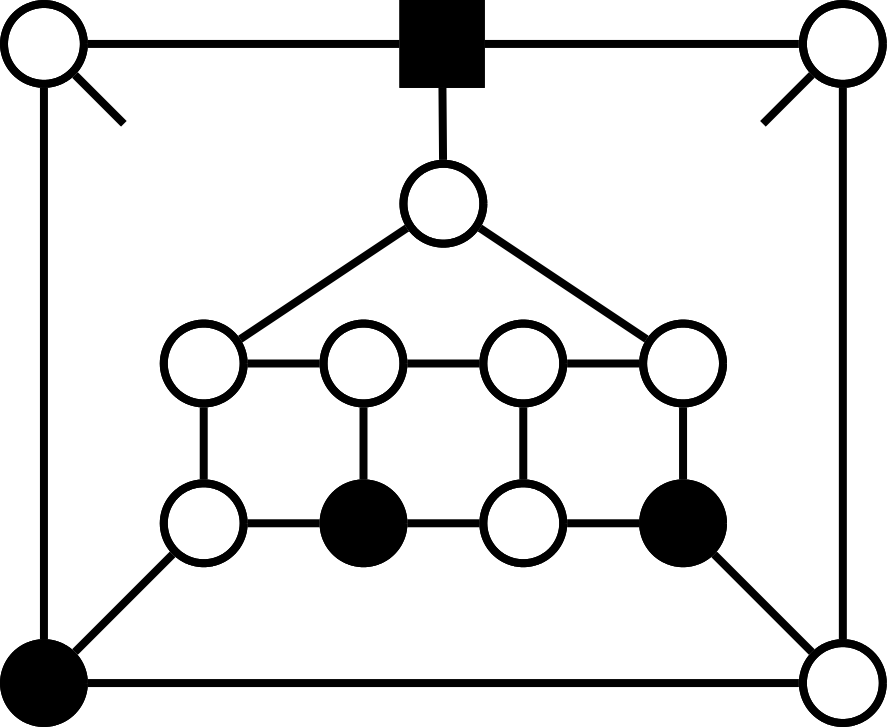}
\subcaption{Period 10}
\end{minipage}
\caption{Triangle-free graphs with minority rule at $v_1$}\label{tfloop}
\end{figure}

Here we give examples of triangle-free graphs which attain periods other than $1$, $2$ or $4$; by \theorem{tfree} such graphs must have loops. The additional periods possible when loops are permitted are $3$, $6$, $8$, $10$ and $12$. By \theorem{mingame}, period $12$ is not possible if $v_1$ obeys an anti-threshold rule. In fact the other periods are possible even with this restriction, as shown in Figure~\ref{tfloop}. Again all neighbourhoods are odd and every vertex obeys the majority rule except for $v_1$ (indicated by the square) which obeys the minority rule; each loop is indicated by a short line leaving its vertex. While \theorem{loops} applies even if triangles which do not meet $v_1$ are permitted, in fact all possible periods can be obtained without triangles anywhere in the graph.

\FloatBarrier\subsection{More general rules at $v_1$}

Finally we give an example of a triangle-free graph on which period $12$ is attained, together with a general construction to show that any period is possible without the restriction that $v_1$ is not in a triangle. In each case a more complicated rule is required at $v_1$. In Figure~\ref{per12}, $v_1$ takes state $+1$ at time $t+1$ if and only if $\abs{N_1\cap U_t}\in\{0,2,3,4\}$, and every other vertex follows the majority rule.

\begin{figure}[!h]
\begin{minipage}[b]{.5\linewidth}
\centering\includegraphics[scale=0.25]{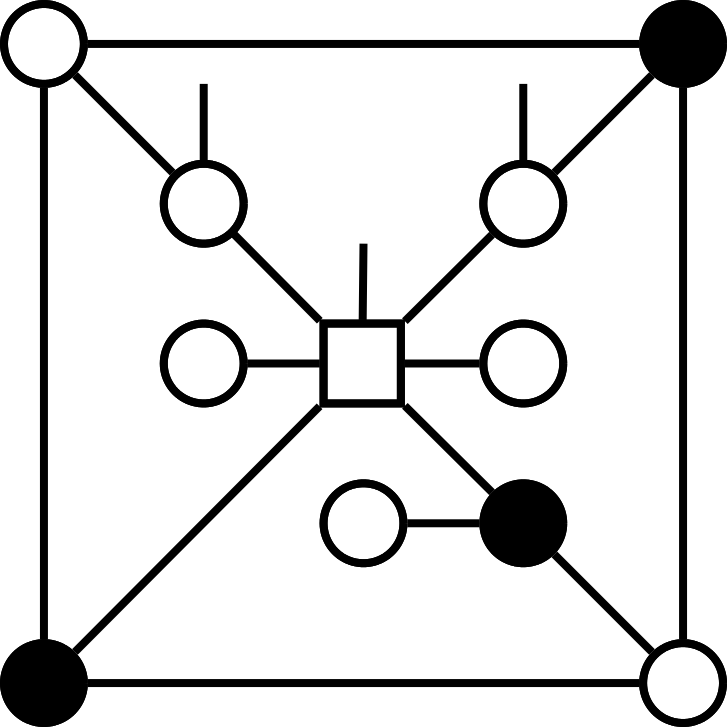}
\subcaption{Period 12}\label{per12}
\end{minipage}%
\begin{minipage}[b]{.5\linewidth}
\centering\includegraphics[scale=0.25]{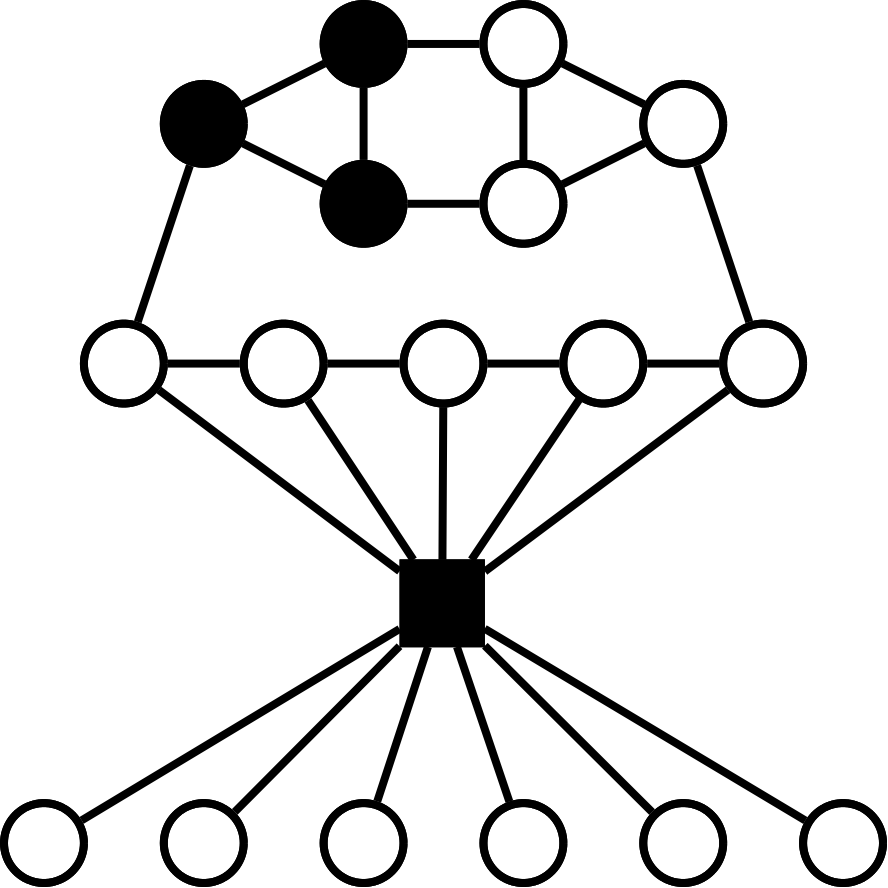}
\subcaption{$G_6$}\label{anyper}
\end{minipage}
\caption{Graphs with more general rules at $v_1$}
\end{figure}

To show that any period except $3$ may be realised on a graph with no restriction on triangles (even if loops are not permitted), we define the graph $G_k$, for $k\geq 2$, on vertex set $\seq v{2k+8}$ as follows. $v_1$ has neighbours $\seq[2]v{2k+2}$, of which $\seq[2]v{k+2}$ have no other neighbours and $\seq[k+3]v{2k+2}$ induce a path. The remaining vertices form two triangles connected by a pair of edges, with two further edges between the ends of the path and the triangles. Figure~\ref{anyper} shows $G_6$. Start from the state shown, \ie $U_0=\{v_1,v_{2k+3},v_{2k+4},v_{2k+5}\}$. Setting $v_1$ to have state $+1$ at time $t+1$ if and only if $\abs{N_1\cap U_t}\in\{0,1,k+1,\ldots,2k\}$ gives period $2k+1$, whereas setting $v_1$ to have state $+1$ at time $t+1$ if and only if $\abs{N_1\cap U_t}\in\{0,1,k+1,\ldots,2k-1\}$ gives period $2k$. Thus any period of at least $4$ occurs on some graph in this sequence, and of course periods $1$ and $2$ can be obtained even with the majority rule at $v_1$. The final case of period $3$ is not possible without loops, and Figure~\ref{three} shows that it is possible if loops are permitted.

\begin{prop}For any graph $G$ without a loop at $v_1$ and for any rule at $v_1$ which is a function of the states of its neighbours, period $3$ is not possible.\end{prop}
\begin{proof}Suppose not, and start the system in a recurrent state of period $3$, so that $U_0\to U_1\to U_2\to U_0$. If $v_1$ is in the same state at every time then $U^*_{2t}$ is monotonic, so constant, contradicting period $3$. So without loss of generality (by swapping states and shifting if necessary), $c_t=+1$ if $t\equiv 0,1$ (mod $3$) and $c_t=-1$ otherwise. Now $\uxle02\subseteq U^*_1$. We must have $\uxd[\neq]12$, since $c_2\neq c_0$. But if $v_i\in\uxm12$ then $\niu1<\niu0$, contradicting $\uxle01$. 
\end{proof}

\FloatBarrier\section{Acknowledgements}

Both authors acknowledge support from the European Union through funding under FP7-ICT-2011-8 project HIERATIC (316705), and the first author also acknowledges support from the European Research Council (ERC) under the European Union's Horizon 2020 research and innovation programme (grant agreement No.\ 639046).

\end{document}